\documentclass[oneside,english]{amsart}
\usepackage[T1]{fontenc}
\usepackage[latin9]{inputenc}
\usepackage{amsthm}
\usepackage{mathrsfs}
\usepackage{amssymb}
\usepackage{esint}

\makeatletter
\numberwithin{equation}{section}
\numberwithin{figure}{section}
\usepackage{enumitem}		
  \theoremstyle{plain}
  \newtheorem*{thm*}{\protect\theoremname}
\theoremstyle{plain}
\newtheorem{thm}{\protect\theoremname}
  \theoremstyle{plain}
  \newtheorem{lem}[thm]{\protect\lemmaname}
  \theoremstyle{plain}
  \newtheorem{cor}[thm]{\protect\corollaryname}
  \theoremstyle{plain}
  \newtheorem{conjecture}[thm]{\protect\conjecturename}
\newtheorem{proposition}[thm]{\protect\propname}
  \theoremstyle{plain}
\newtheorem*{cor*}{\protect\corollaryname}
  \theoremstyle{plain}

\usepackage{fullpage}

\makeatother

\usepackage{babel}
  \providecommand{\conjecturename}{Conjecture}
  \providecommand{\corollaryname}{Corollary}
  \providecommand{\lemmaname}{Lemma}
  \providecommand{\theoremname}{Theorem}
\providecommand{\propname}{Proposition}

\newcommand{\alg}{\operatorname{Alg}}

\newcommand{\norm}[1]{\left\Vert#1\right\Vert}
\newcommand{\paren}[1]{\left(#1\right)}
\newcommand{\abs}[1]{\left|#1\right|}

\newcommand{\R}{\mathbb{R}}
\newcommand{\C}{\mathbb{C}}

\begin{document}

\title{Regularity of Polynomials in Free Variables}

\author{Ian Charlesworth and Dimitri Shlyakhtenko}

\thanks{Research supported by NSF grants DMS-1161411 and DMS-1500035, and NSERC award PGS-6799-438440-2013.}
\begin{abstract}
We show that the spectral measure of any non-commutative polynomial
of a non-commutative $n$-tuple cannot have atoms if the free entropy
dimension of that $n$-tuple is $n$ (see also work of Mai, Speicher, and Weber).
Under stronger assumptions on the $n$-tuple, we prove that the spectral measure is not singular, and measures of intervals surrounding any point
may not decay slower than polynomially as a function of the interval's length.
\end{abstract}
\maketitle

\section{Introduction.}

It was shown in \cite{shlyakhtenko-skoufranis:atoms} that if $y_{1},\dots,y_{n}$
are free self-adjoint non-commutative random variables and $P$ is
any self-adjoint non-commutative polynomial in $n$ indeterminates,
then the spectral measure of $y=P(y_{1},\dots,y_{n})$ cannot have
atoms, unless $P$ is constant. This in particular implies that a
polynomial of $n$ semicircular elements cannot have atoms. When applied
to random matrix theory, this shows that if $Y_{1},\dots,Y_{n}$ are
independent Gaussian Random Matrices, then the eigenvalues of $P(Y_{1},\dots,Y_{N})$
exhibit a very weak form of repulsion: the expected proportion of
the number of eigenvalues in any interval $[a,b]$ goes to zero with
$N$ as the size of the interval shrinks. 

Statements of this kind can be considered to be part of the study
of consequences of regularity assumptions on non-commutative transformations,
which is of significant interest due to, for example, the results
of \cite{guionnet-shlyakhtenko:transport}. Since \cite{shlyakhtenko-skoufranis:atoms}, there
have been important advances in this direction. Very recently, Figalli
and Guionnet \cite{guionnet-figalli:transport} have used transport
maps to give a full picture of the behavior of eigenvalues of $P(Y_{1},\dots,Y_{n})$
under the assumption that $P$ is close to the identity transformation.
Finally, Mai, Speicher and Weber \cite{msw-noatoms} have been able able to obtain a regularity result for  $p(y_{1},\dots,y_{n})$ in absence of freeness assumption and beyond the perturbative regime.  While initially requiring a stronger assumption, they were also able to give a proof of the following theorem using their methods after seeing an early version of our note; our proof appears in the next section.
\begin{thm*}
Assume that Voiculescu's free entropy dimension \cite{dvv:entropy2,dvv:entropy5} 
$\delta^{*}(y_{1},\dots,y_{n})=n$. Then for any self-adjoint non-constant non-commutative
polynomial $P$, the spectral measure of $y=P(y_{1},\dots,y_{n})$
has no atoms. In particular, $\delta^{*}(y)=1$. 
\end{thm*}
The assumption of the theorem is satisfied in each of the following
cases:
\begin{enumerate}
\item If $y_{j}$ are freely independent and each has non-atomic distribution
(compare \cite{shlyakhtenko-skoufranis:atoms}). Indeed, in this case
$\delta^{*}(y_{j})=1$ and so by free independence $\delta^{*}(y_{1},\dots,y_{n})=\sum\delta^{*}(y_{j})=n$.
In particular, this holds if $y_{1},\dots,y_{n}$ are a free semicircular
family.
\item If $\Phi^{*}(y_{1},\dots,y_{n})<+\infty$.
\item If Voiculescu's microstates free entropy \cite{dvv:entropy2} $\chi(y_{1},\dots,y_{n})$
or non-microstates free entropy \cite{dvv:entropy5} $\chi^{*}(y_{1},\dots,y_{n})$
are finite, or if the microstates free entropy dimension $\delta(y_{1},\dots,y_{n})=n$.
Indeed, in all of these cases $\delta^{*}(y_{1},\dots,y_{n})=n$ (making
use of the inequality between microstates and non-microstates free
entropy proved in \cite{bcg}).
\end{enumerate}
It is not hard to see that the hypothesis of this theorem is optimal.
Indeed, for a free $n$-tuple $y_{1},\dots,y_{n}$ in which the law
of $y_{1}$ has a single atom of weight $\lambda$ and the laws of
$y_{2},\dots,y_{n}$ are non-atomic, one has $\delta^{*}(y_{1},\dots,y_{n})=n-\lambda^{2}$.

We also prove that if $\delta^{*}(y_{1},\dots,y_{n})=n$, then there
can be no algebraic relations between convergent power series in $y_{1},\dots,y_{n}$.
This was proved under stronger assumptions by Dabrowski (see Lemma
37 in \cite{yoann:SPDE}); a version for polynomials appears in \cite{msw-noatoms}.

Our proof uses ideas from $L^{2}$-homology going back to the second author's joint
work with Connes \cite{cs} to give an alternate proof of the key
lemma in \cite{msw-noatoms}. The remainder
of the proof is essentially the same as in \cite{msw-noatoms}.

In Section \ref{sec:algent}, we demonstrate that if $y \in M$ has no atoms and is algebraic (i.e., has spectral measure with algebraic Cauchy transform), then $\chi(y) > -\infty$.
The argument hinges on work of Anderson and Zeitouni in \cite{anderson-zeitouni:06}, and in particular Theorem 2.9 of that paper, which controls the density of the spectral measure of such a variable $y$.
This leads us to the following result:
\begin{cor*}
Assume that $y_1, \ldots, y_n$ are free, algebraic, and $\chi(y_j) > -\infty$ for $1 \leq j \leq n$.
Then if $y = P(y_1, \ldots, y_n)$ with $P$ a non-constant self-adjoint polynomial, $\chi(y) > -\infty$.
\end{cor*}

In the final section of this paper, we show that under some stronger assumptions, such as the existence of a dual system, the spectral measure of $y = P(y_1, \ldots, y_n)$ cannot be purely singular with respect to Lebesgue measure, and in certain cases such as when $y$ is a monomial, must be absolutely continuous.
Additionally, we demonstrate that under the same assumption of a dual system, the spectral measure of $y$ must not decay at a rate slower than polynomially in the neighbourhood of any fixed point.

\section{Absence of Atoms and Zero Divisors.}

\subsection{Non-commutative polynomials, power series and difference quotients. }

If $M$ is an algebra, $u,v\in M$ and $T$ is an element of an $M,M$-bimodule,
then we use the notation $(u\otimes v)\#T=uTv$. We view $M\otimes M$
as an $M,M$-bimodule as follows: for $a\otimes b\in M\otimes M$,
$u,v\in M$, set 
\[
(u\otimes v)\#(a\otimes b)=ua\otimes bv.
\]

We denote by $\mathbb{C}[X_{1},\dots,X_{n}]$ the algebra of non-commuting
polynomials in $n$ indeterminates $X_{1},\dots,X_{n}$. We denote
by $\partial_{j}$ Voiculescu's difference quotient derivations \cite{dvv:entropy5}
with values in $\mathbb{C}[X_{1},\dots,X_{n}]\otimes\mathbb{C}[X_{1},\dots,X_{n}]$,
determined by $\partial_{j}X_{i}=\delta_{i=j}1\otimes1$.

We denote by $\mathbb{C}\langle X_{1},\dots,X_{n};R\rangle$ the completion
of $\mathbb{C}[X_{1},\dots,X_{n}]$ in the norm
\[
\left\Vert \sum_{N}\sum_{j_{1},\dots,j_{N}=1}^{n}\alpha_{N;j_{1},\dots,j_{N}}X_{j_{1}}\cdots X_{j_{N}}\right\Vert _{R}=\sum_{N}\sum_{j_{1},\dots,j_{N}=1}^{n}\left|\alpha_{N;j_{1},\dots,j_{N}}\right|R^{N}.
\]
We will also write
\[
\mathcal{A}[X_{1},\dots,X_{n};R]=\bigcup_{R'>R}\mathbb{C}\langle X_{1},\dots,X_{n};R\rangle
\]
for the algebra of power series with norm strictly bigger than $R$.
Of course, if $y_{1},\dots,y_{n}$ are elements of any Banach algebra
$M$ and $\max_{j}\Vert y_{j}\Vert<R$, then there is a unique homomorphism
from $\mathcal{A}[X_{1},\dots,X_{n};R]$ which sends $X_{j}$ to $y_{j}$,
$j=1,\dots,n$. For $P\in\mathcal{A}[X_{1},\dots,X_{n};R]$, we write
$P(y_{1},\dots,y_{n})$ for the image of $P$ under this homomorphism.

It is not hard to see that $\partial_{j}$ extends to an (unbounded)
derivation
\[
\partial_{j}:\mathcal{A}[X_{1},\dots,X_{n};R]\to\mathcal{A}[X_{1},\dots,X_{n};R]\hat{\otimes}\mathcal{A}[X_{1},\dots,X_{n};R],
\]
where $\hat{\otimes}$ denotes the projective tensor product. Likewise,
if $B$ is any Banach bimodule over $\mathbb{C}[X_{1},\dots,X_{n}]$
for which the right and left actions of $X_{j}$ are bounded operators
of norm at most $R$, then $\#$ extends to an action of $\mathcal{A}[X_{1},\dots,X_{n};R]\hat{\otimes}\mathcal{A}[X_{1},\dots,X_{n};R]$
on $B$. In particular, if $Q$ is any element of the space $\mathcal{L}^{1}(L^{2}(M))$
of trace-class operators on $M$, then $P\mapsto\partial_{j}(P)\#Q$
extends to a map defined on $\mathcal{A}[X_{1},\dots,X_{n};R]$. 

Let $\mathcal{N}:\mathcal{A}[X_{1},\dots,X_{n};R]\to\mathcal{A}[X_{1},\dots,X_{n};R]$
be given by $\mathcal{N}P=\sum_{j=1}^{n}\partial_{j}P\#X_{j}$. Thus
$\mathcal{N}$ is a kind of number operator, multiplying a monomial
by its degree. 

We note that the formula
\begin{equation}
\phi_{t}(P)=\sum_{k\geq0}\mathcal{N}^{k}(2\pi i\ t)^{k}/k!\label{eq:phi}
\end{equation}
gives rise to an automorphism of $\mathcal{A}[X_{1},\dots,X_{n};R]$
which multiplies a monomial of degree $d$ by $\exp(2\pi i\ t\ d)$.

\subsection{Algebraic relations and Hochschild cycles.}

The following lemma shows that algebraic relations in the algebra
generated by $y_{1},\dots,y_{n}$ produce Hochschild cycles: 
\begin{lem}
\label{lemma:derivations} Suppose that $P\in\mathbb{\mathcal{A}}[X_{1},\dots,X_{n};R]$
and suppose that for some self-adjoint elements $y_{1},\dots,y_{n}\in(M,\tau)$
with $\max\Vert y_{j}\Vert<R$ and $u,v\in M$, $(u\otimes v)\#(P(y_{1},\dots,y_{n})\otimes1-1\otimes P(y_{1},\dots,y_{n}))=0$.
Let $T_{i}=(\partial_{i}P)(y_{1},\dots,y_{n})\in M\otimes M$. Then
\[
\sum(u\otimes v)\#T_{i}\#(y_{i}\otimes1-1\otimes y_{i})=0.
\]
Thus if we put for $y,y'\in M$, $Jy^{*}J(a\otimes b)J(y')^{*}J=ay\otimes y'b$,
the identity 
\[
\sum[uT_{i}v,Jy_{i}^{*}J]=0
\]
holds. \end{lem}
\begin{proof}
Note that 
\[
(u\otimes v)\#\sum T_{i}(y_{i}\otimes1-1\otimes y_{i})=(u\otimes v)\#(P(y_{1},\dots,y_{n})\otimes1-1\otimes P(y_{1},\dots,y_{n}))=0,
\]
which gives the first equation. The second equation follows from the
definition of the action of $Jy_{i}^{*}J$ and commutation of the
variables $u,v$ with the variables $Jy_{j}^{*}J$, $j=1,\dots,n$. 
\end{proof}
Let $\delta^{*}(y_{1},\dots,y_{n})$ be Voiculescu's non-microstates
free entropy dimension. Denote by $FR=FR(L^{2}(M))$ the space of
finite-rank operators on $L^{2}(M)$. Let $\Delta(y_{1},\dots,y_{n})$
be the quantity introduced in \cite{cs}: 
\[
\Delta(y_{1},\dots,y_{n})=n-\dim_{M\bar{\otimes}M}\overline{\{(T_{1},\dots,T_{n})\in FR(L^{2}(M)):\sum[T_{j},Jy_{j}^{*}J]=0\}},
\]
where the closure is taken in the Hilbert-Schmidt norm. 

Let us denote by $\mathcal{L}^{1}=\mathcal{L}^{1}(L^{2}(M))$ the
space of trace-class operators on $L^{2}(M)$. Let
\[
\Delta_{1}(y_{1},\dots,y_{n})=n-\dim_{M\bar{\otimes}M}\overline{\{(T_{1},\dots,T_{n})\in\mathcal{L}^{1}(L^{2}(M)):\sum[T_{j},Jy_{j}^{*}J]=0\}},
\]
where once again the closure is taken in the Hilbert-Schmidt norm.
Clearly, $\Delta_{1}\leq\Delta.$ By \cite[Lemma 4.1 and Theorem 4.4]{cs}
we have inequalities 
\[
\delta^{*}(y_{1},\dots,y_{n})\leq\Delta(y_{1},\dots,y_{n})\leq n.
\]
The same theorem (with exactly the same proof) shows that also $\delta^{*}(y_{1},\dots,y_{n})\leq\Delta_{1}(y_{1},\dots,y_{n})$
(indeed, the only change is in the Lemma 4.2, which obviously works
if we replace\emph{ $FR$} by $\mathcal{L}^{1}$). Thus 
\[
\delta^{*}(y_{1},\dots,y_{n})\leq\Delta_{1}(y_{1},\dots,y_{n})\leq\Delta(y_{1},\dots,y_{n})\leq n.
\]

In particular, if $\delta^{*}(y_{1},\dots,y_{n})=n$, then $\Delta(y_{1},\dots,y_{n})=n$,
which implies that 
\[
\overline{\{(T_{1},\dots,T_{n})\in\mathcal{L}^{1}(L^{2}(M)):\sum[T_{j},Jy_{j}^{*}J]=0\}}=\{0\}.
\]
We record this as:
\begin{cor}
\cite{cs}\label{thrm:cs} Let $M=W^{*}(y_{1},\dots,y_{n})$. Assume
that $\delta^{*}(y_{1},\dots,y_{n})=n$. Then the only trace-class
operators $R_{j}\in\mathcal{L}^{1}(L^{2}(M,\tau))$, $j=1,\dots,n$,
that satisfy 
\[
\sum[R_{i},Jy_{i}^{*}J]=0
\]
are $R_{1}=R_{2}=\cdots=R_{n}=0$. 
\end{cor}

\subsection{Absence of zero divisors and algebraic relations.}
\begin{thm}
\label{thm:noZeroDivisors}Let $y_{1},\dots,y_{n}\in M$ be self-adjoint
elements for which $\delta^{*}(y_{1},\dots,y_{n})=n$. Assume that
for some $P\in\mathbb{C}[X,\dots,X_{n}]$ and some projection $0\neq p\in W^{*}(y_{1},\dots,y_{n})$,
$P(y_{1},\dots,y_{n})\cdot p=0$. Then $P=0$. \end{thm}
\begin{proof}
Assume for contradiction that $P\neq0$; obviously, $P$ is not constant.
We may further assume that $P$ is the smallest degree non-constant
polynomial with the property that $P(y_{1},\dots,y_{n})\cdot p=0$. 

Because $M$ is a tracial von Neumann algebra, the left and right
support projections of $P(y_{1},\dots,y_{n})$ are equal; thus for
some projection $q_{1}$ with the same trace as $p$, $q_{1}P(y_{1},\dots,y_{n})=0$. 

We apply Lemma \ref{lemma:derivations} with $u=q_{1}$, $v=p$ to
conclude that if we set $R_{j}=(u\otimes1+1\otimes v)\#(\partial_{j}P)(y_{1},\dots,y_{n})$,
then $\sum[q_{1}R_{j}p,Jy_{j}^{*}J]=0$. Then Corollary \ref{thrm:cs}
implies that $R_{j}=0$ for all $j$; in other words,
\begin{equation}
(q_{1}\otimes1+1\otimes p)\partial_{j}(P)(y_{1},\dots,y_{n})=0.\label{eq:partialszero-1}
\end{equation}
Let 
\[
\Delta_{j,q_{1}}(P)=(\tau\otimes1)\left((q_{1}\otimes1)\#\partial_{j}(P)\right)\in\mathbb{C}[X_{1},\dots,X_{n}].
\]
Then $\Delta_{j,q_{1}}(P)(y_{1},\dots,y_{n})p=0$, but the degree
of $\Delta_{j,q_{1}}(P)$ is smaller than that of $P$. It follows
that $\Delta_{j,q_{1}}P=0$ or that $P$ is linear. In the former
case, let $X_{i_{1}}\cdots X_{i_{s}}$be a highest-degree term in
$P$ with a nonzero coefficient, $\alpha$. Then applying our proof
recursively, we obtain projections $q_{1},q_{2},\dots,q_{s}$ of trace
equal to $p$, so that $(\Delta_{i_{s},q_{s}}\circ\cdots\circ\Delta_{i_{1},q_{1}})(P)=\alpha\prod\tau(q_{j})=\alpha\tau(p)^{s}$,
contradicting $\Delta_{i_{1},q_{1}}(P)=0$. If $P$ is linear, then
$\Delta_{j,q_{1}}(P)$ does not depend on $X_{1},\dots,X_{n}$ and
so $\Delta_{j,q_{1}}(P)(y_{1},\dots,y_{n})=0$ implies that $P$ is
constant.\end{proof}
\begin{thm}
Let $y_{1},\dots,y_{n}\in M$ be self-adjoint elements satisfying
$\max_{j}\Vert y_{j}\Vert<R$ and $\delta^{*}(y_{1},\dots,y_{n})=n$.
Assume that for some $P\in\mathcal{A}[X_{1},\dots,X_{n};R]$, $P(y_{1},\dots,y_{n})=0$.
Then $P=0$.\end{thm}
\begin{proof}
Assume for contradiction that $P\neq0$. Obviously, $P$ is not constant. 

We apply Lemma \ref{lemma:derivations} with $u=v=1$: if $R_{j}=(\partial_{j}P)(y_{1},\dots,y_{n})$,
then $\sum[R_{j},Jy_{j}^{*}J]=0$. Thus Corollary \ref{thrm:cs} implies
that $R_{j}=0$ for all $j$; in other words,
\[
\partial_{j}(P)(y_{1},\dots,y_{n})=0
\]
 It follows that also $\mathcal{N}P=\sum_{j}\partial_{j}P\#X_{j}$
satisfies 
\[
\mathcal{N}P(y_{1},\dots,y_{n})=0.
\]
Hence if $\phi_{t}$ is given by (\ref{eq:phi}), then 
\[
\phi_{t}(P)(y_{1},\dots,y_{n})=0.
\]
Thus
\[
\int_{0}^{1}\exp(-2\pi i\ m\ t)\phi_{t}(P)(y_{1},\dots,y_{n})=0,\qquad\forall m>0.
\]
Thus if we write $P^{(m)}$ for the sum of degree $m$ terms of $P$,
then $P^{(m)}(y_{1},\dots,y_{n})=0$, for each $m>0$. But then we
may use Theorem \ref{thm:noZeroDivisors} with $p=1$ to conclude
that $P^{(m)}=0$ for all $m>0$ and that $P=0$.
\end{proof}

\subsection{Further questions.}

Assume now that $z_{1},\dots,z_{m}\in M$ are self-adjoint, but make
no assumptions on whether they generate $M$ or not. Put
\[
\Delta_{M}(z_{1},\dots,z_{m})=m-\dim_{M\bar{\otimes}M^{o}}\overline{\{(T_{1},\dots,T_{m})\in FR(L^{2}(M))^{m}:\sum[T_{j},Jz_{j}^{*}J]=0\}}^{HS}.
\]

\begin{conjecture}
Let $y_{1},\dots,y_{n}\in M$ be self-adjoint elements so that $M=W^{*}(y_{1},\dots,y_{n}).$\\
(a) Let $P_{1},\dots,P_{m}\in\mathbb{C}[X_{1},\dots,X_{n}]$ be polynomials.
Let $z_{j}=P_{j}(y_{1},\dots,y_{n})$, $j=1,\dots,m$. Assume finally
that $M=W^{*}(y_{1},\dots,y_{n}).$ Assume that $\Delta_{M}(y_{1},\dots,y_{n})=n$.
Then exactly one of the following statements holds:
\begin{enumerate}[label={({\roman*})}]
\item $\Delta_{M}(z_{1},\dots,z_{m})\leq m-1$; or
\item $\Delta_{M}(z_{1},\dots,z_{m})=m$.
\end{enumerate}
Moreover, if (b) holds and $\chi^{*}(y_{1},\dots,y_{n})>-\infty$ then
also $\chi^{*}(z_{1},\dots,z_{m})>-\infty$. \\
(b) Let $N\geq1$ be fixed, and let $P_{ij}\in\mathbb{C}[X_{1},\dots,X_{n}]$,
$1\leq i,j,\leq N$. Let $y=(P_{ij}(y_{1},\dots,y_{n}))_{ij}\in M_{N\times N}(M)$,
and assume that $y=y^{*}$. Then any atom of the spectral measure
of $y$ must have normalized trace in the set $\{0,1/N,\dots,N/N\}$.\end{conjecture}
Part \emph{(b)} above is known to hold in the special case that $y_1, \ldots, y_n$ are a free semi-circular family; see \cite{shlyakhtenko-skoufranis:atoms}.

\section{Algebraicity and Free Entropy.}
\label{sec:algent}
\begin{lem}
\label{lem:lpentropy}
Suppose that $y = y^* \in (M, \tau)$ is an element of a von Neumann probability space.
Further assume that the spectral distribution of $y$ is Lebesgue absolutely continuous, with density $f$.
Lastly, suppose that $f \in L^p(\R)$ for some $p > 1$.
Then $\chi(y) > -\infty$.
\end{lem}

\begin{proof}
Using interpolation and the fact that $\norm{f}_1 = 1$, we may assume that $p < 2$.
For a single variable, we know by \cite{dvv:entropy1} that the free entropy satisfies
$$
\chi(x) = \iint \log\abs{s-t}\,d\mu(s)\,d\mu(t) + C,
$$
for $C$ constant, so it suffices to bound the integral.
Let $L(t) := \chi_{(-4M, 4M)}\log\abs t$, where $M$ is large enough that the support of $\mu$ is contained in $(-M, M)$.
Then we have for $s \in \R$,
$$
f(s)\int f(t)\log\abs{s-t}\,dt
= f(s)\int f(t)L(s-t)\,dt
= f(s)(f\star L)(s).
$$
We compute:
\begin{align*}
\abs{\iint \log\abs{t-s}\,d\mu(t)\,d\mu(s)}
&= \abs{\int f(s) \int f(t)\log\abs{s-t}\,dt\,ds}\\
&= \abs{\int f(s)(f\star L)(s)\,ds}\\
&\leq \norm{f\cdot(f\star L)}_1\\
&\leq \norm{f}_p\norm{f\star L}_q,
\end{align*}
where $1 = \frac1p+\frac1q$, by H\"older's inequality; note that $q > 2$ since $p < 2$.
It suffices, now, to show that $f\star L \in L^q(\R)$; for this, we appeal to Young's inequality.
Indeed, if $s = \frac q2 > 1$, then $1 + \frac 1q = \frac1p+\frac1s$ and so we have
$$\norm{f \star L}_q \leq \norm{f}_p\norm{L}_s.$$
Yet $\norm{L}_s < \infty$ for any $1 \leq s < \infty$, so we conclude $\abs{\chi(x)} < \infty$.
\end{proof}

Suppose that $y = y^* \in M$ is algebraic, in the sense that its spectral measure $\mu$ has algebraic Cauchy transform.
Assume further that $\mu$ has no atoms.
It was shown in \cite[Theorem 2.9]{anderson-zeitouni:06} that such a measure $\mu$ has density $f$ which fails to exist at at most finitely many points, and for some $d > 0$ and every $a \in \R$ satisfies
$$\lim_{x\to a}(x-a)^{1-d}f(x) < \infty.$$
In particular, if $1 < p < (1-d)^{-1}$, we have $f \in L^p(\R)$ and so Lemma $\ref{lem:lpentropy}$ leads us to conclude $\chi(y) > - \infty$.

\begin{cor}
Assume that $y_1, \ldots, y_n$ are free, algebraic, and $\chi(y_j) > -\infty$ for $1 \leq j \leq n$.
Then if $y = P(y_1, \ldots, y_n)$ with $P$ a non-constant polynomial, $\chi(y) > -\infty$.
\end{cor}

\begin{proof}
It was shown by Anderson in \cite{anderson:14} that the freeness and algebraicity of $y_1, \ldots, y_n$ is enough to guarantee that $y$ is algebraic.
Freeness and the fact that $\chi(y_j) > -\infty$ together with Theorem \ref{thm:noZeroDivisors} allow us to conclude that the spectral measure of $y$ has no atoms.
Following the above discussion, we conclude $\chi(y) > -\infty$.
\end{proof}

It is tempting to conjecture that if $y_1,\dots,y_n$ satisfy $\chi (y_1,\dots,y_n)>-\infty$ (where we use Voiculescu's microstates free entropy, \cite{dvv:entropy2}) and $y=p(y_1,\dots,y_n)$ is a non-constant polynomial, then also $\chi(y)>-\infty$.  We point out that this is true assuming that $p$ is sufficiently close to a polynomial of degree $1$.  For a non-commutative polynomial $p$ let us write $a_{i_1,\dots,i_n}$ for the coefficient of the monomial $y_{i_1}\cdots y_{i_n}$ in $p$:

\begin{proposition}
Assume that $\chi(y_1,\dots,y_n)>-\infty$.  Fix $d$ and $n$.  Then there exists an $\epsilon>0$ so that if  $p(y_1,\dots,y_n)$ is a self-adjoint polynomial of degree $d$ for which $$\sum_{k\geq 2} \sum_{i_1,\dots,i_k=1}^n | a_{i_1,\dots,i_k} |^2 < \epsilon \sum_{j} |a_j|^2$$ and $y=p(y_1,\dots,y_n)$, then $$\chi(y)>-\infty.$$
\end{proposition}

\begin{proof}
We note first that we can assume that $p$ has no constant term, since this has no effect on the free entropy of $y$.  Let $\delta = \sum_j |a_j|^2$.  We note that by Voiculescu's change of variable formula, $\chi(y_1,\dots,y_n)$ is unchanged if we replace $y_j$ by $\sum_i T_{ij} y_i$, where $(T_{ij})$ is any orthogonal matrix.  We may thus assume that $a_{j}=0$ unless $j=1$ and that $|a_1|^2 = \delta$.  By multiplying $p$ by $\delta^{-1}$ (which changes the free entropy of $y$ by an finite additive constant) we may assume that $a_1 = 1$ and $\sum_{i_1,\dots,i_k=1}^n | a_{i_1,\dots,i_k} |^2 < \epsilon$.

Next, we set $q_1(y_1,\dots,y_n) = p$ and let $q_j (y_1,\dots,y_n) = y_j$ for $j=2,\dots,n$.  Then by the  implicit function theorem for non-commutative power series (see e.g. \cite{guionnet-shlyakhtenko:transport}) we know that as long as $\epsilon$ is sufficiently small, there exist non-commutative power series $f_1,\dots,f_n$ so that 
\begin{align*}
t_j & = q_j (f_1(t_1,\dots,t_n),\dots,f_n(t_1,\dots,t_n))  \\ s_j &=  f_j (q_1(s_1,\dots,s_n),\dots,q_n(s_1,\dots,s_n)) 
\end{align*} for all $j$ and any operators $t_1,\dots,t_n$ and $s_1,\dots,s_n$ of norm at most $1.1 \max_{j} \Vert y_j \Vert$.  Voiculescu's change of variables formula for free entropy \cite{dvv:entropy2} implies that $$\chi(q_1(y_1,\dots,y_n),\dots,q_n(y_1,\dots,y_n))>-\infty.$$  But then $$
-\infty < \chi(y,y_2,\dots,y_n) \leq \chi(y) + \chi(y_2) + \cdots + \chi(y_n) $$ which implies that $\chi(y)>-\infty$.
\end{proof}

\section{Properties of Spectral Measures.}
\subsection{Non-singularity.}
For $A$ an algebra, we denote the flip operation on $A$ as follows: given $a, b \in A$, set $(a\otimes b)^\sigma = b\otimes a$, and extend this linearly to $A \otimes A$.
We also maintain the notation $(a\otimes b)\# z = azb$ for $a, b, z \in A$.

\begin{lem}
\label{lemma:sing-commute}
Let $x\in B(H)$ be a self-adjoint operator, and assume that the spectral measure of $x$ is not Lebesgue absolutely continuous.  Then there exists a sequence $T_n$ of finite-rank operators satisfying the following properties: (i) $0\leq T_n \leq 1$; (ii) for a nonzero spectral projection $p$ of $x$, $T_n \to p$ weakly; (iii) $\Vert [T_n,x]\Vert_1 \to 0$.  Moreover, if the absolutely continuous component of the spectral measure of $x$ is zero, then $p$ can be chosen to be $1$.
\end{lem}

\begin{proof}
Let $p$ be a nonzero spectral projection of $x$ corresponding to a subset of $\mathbb R$ having Lebesgue measure zero.  By replacing $x$ with $pxp$, we may assume that $x$ has singular spectral measure.  In this case, by a result of Voiculescu (see Theorem 4.5 in \cite{dvv:norm-ideal-perturb}; note that as remarked at the bottom of page 5 of that article, $\mathscr{C}_1^- = \mathscr{C}_1$ and so $k_1^- = k_1$; in this particular case the result essentially goes back to Kato \cite{kato:perturb}), there exist $0\leq T_n \leq 1_{pH}$ with the property that $T_n \to 1_{pH}$ weakly in $B(pH)$ and $\Vert [x,T_n]\Vert_1 \to 0$.
\end{proof}

\begin{lem}
\label{lemma:zero-der}
Let $\operatorname{Alg}(y_1, \ldots, y_n) = A$ be a $*$-algebra with $y_i = y_i^*$, $y \in A$ a polynomial, and $\tau$ a faithful trace on $A$.
Suppose that $y_1, \ldots, y_n$ are algebraically free, and for each $1\leq i \leq n$ let $\partial_i: A \to A\otimes A$ be the derivation given by $\partial_i(y_j) = \delta_{j=i}1\otimes1$.
Then $y \in \operatorname{Alg}(y_1, \ldots, \hat y_i, \ldots, y_n)$ if and only if
$$(\partial_i y)^\sigma\#y^* = 0.$$
Moreover, if $(\partial_iy)^\sigma\#y^* = 0$ for every $1 \leq i \leq n$, then $y \in \C$.
\end{lem}

\begin{proof}
One direction is immediate as $\operatorname{Alg}(y_1, \ldots, \hat y_i, \ldots, y_n) \subseteq \ker\partial_i$.

Let $\mathcal{N}_i : A \to A$ be the number operator associated to $y_i$, the linear map which multiplies each monomial by its $y_i$-degree.
Observe that
$$(\partial_i y)\# y_i = \mathcal{N}_i(y),$$
as each monomial $m$ in $y$ contributes $\sum_{m = ay_ib} (a\otimes b)\# y_i = \mathcal{N}_i(m)$.

Suppose, then, that $(\partial_i y)^\sigma\#y^* = 0$, so $y_i(\partial_i y)^\sigma\#y^* = 0$ as well.
Let $\varphi_\lambda : A \to A$ be the algebra homomorphism given by $\varphi_\lambda(y_i) = \lambda y_i$, $\varphi_\lambda(y_j) = y_j$ for $j \neq i$, which exists as $y_1, \ldots, y_n$ are algebraically free.
We compute the following:
$$
0
= \tau\circ\varphi_\lambda(0)
= \tau\circ\varphi_\lambda\left(y_i(\partial_i y)^\sigma\#y^*\right)
= \tau\circ\varphi_\lambda\left(y^*(\partial_i y)\# y_i\right)
= \tau\circ\varphi_\lambda\left(y^* \mathcal{N}_i(y)\right).
$$

Now, suppose $\deg_{y_i}(y) = d$, and take $x, z \in A$ so that $x$ is $y_i$-homogeneous of degree $d$, $\deg_{y_i}(z) < d$, and $y = x + z$.
Then:
$$
0
= \tau\circ\varphi_\lambda\left( y^*\mathcal{N}_i(y) \right)
= \tau\circ\varphi_\lambda(dx^*x) + \tau\circ\varphi_\lambda(z^*\mathcal{N}_i(y) + x^*\mathcal{N}_i(z))
= d\lambda^{2d}\tau(x^*x) + O(\lambda^{2d-1}).
$$
Thus $d\lambda^{2d}\tau(x^*x) = 0$, and as $\tau$ is faithful, either $x = 0$ (in which case $y = 0$) or $d = 0$ (in which case $\deg_{y_i}(y) = 0$).
Either way, $y \in \operatorname{Alg}(y_1, \ldots, \hat y_i, \ldots, y_n)$.

Repeated application of the above yields the final claim.
\end{proof}

\begin{thm}
\label{prop:nonsing}
Let $y=y^* \in \operatorname{Alg}(y_1,\dots,y_n) \subset W^*(y_1,\dots,y_m)=M$ be a non-constant polynomial in a $W^*$-probability space and assume that there exists a dual system to $y_1,\dots,y_m$, i.e., operators $R_j\in B(L^2(M))$, $1\leq j\leq m$ so that $[R_j,y_k]=\delta_{j=k} P_1$, where $P_1$ is the orthogonal projection onto $1\in L^2(M)$.
Then the spectral measure of $y$ is not singular with respect to Lebesgue measure.
\end{thm}
\begin{proof}
Assume to the contrary that the spectral measure of $y$ is singular with respect to Lebesgue measure.
By Lemma \ref{lemma:sing-commute}, we may choose $0\leq T_n \leq 1$ finite rank operators with $T_n\to 1$ weakly and $\Vert [T_n,y]\Vert_1 \to 0$.  


Fix $x \in M$.
We compute as follows:
\begin{align*}
0
&= \lim_{n\to\infty} \Vert Jx^*J R_i y \Vert_\infty \Vert [T_n, y] \Vert_1 \\
&\geq \lim_{n\to\infty} \Big| \operatorname{Tr}( Jx^*J R_i y [T_n, y]) \Big| \\
&= \lim_{n\to\infty} \Big| \operatorname{Tr}(Jx^*J [y, R_i] yT_n) \Big| \\
&= \lim_{n\to\infty} \Big| \operatorname{Tr}(Jx^*J ((\partial_i y)\# P_1) yT_n) \Big| \\
&= \Big| \operatorname{Tr}( Jx^*J P_1 ((\partial_i y)^\sigma \# y) ) \Big| \\
&= \Big| \tau(((\partial_i y)^\sigma\#y)x) \Big|
\end{align*}
Here we are able to pass to the limit of $T_n$ as $P_1$ is finite rank.

We conclude that $(\partial_i y)^\sigma\#y = 0$ as $\tau$ is faithful and as $x$ was arbitrary.
From Lemma \ref{lemma:zero-der}, we see that $y$ is constant.
\end{proof}

For certain polynomials, we can make a stronger statement: that the spectral measure is in fact Lebesgue absolutely continuous.
To do so, we will need the following modification of Lemma \ref{lemma:zero-der}.

\begin{lem}
\label{lemma:zero-der2}
Let $y = y^* \in \operatorname{Alg}(y_1, \ldots, y_n) \subset W^*(y_1, \ldots, y_m) = M$ be a polynomial in a tracial von Neumann probability space, suppose $y_1, \ldots, y_m$ are algebraically free and have full free entropy dimension, and take $\partial_i : A \to A\otimes A$ as before.
Fix $1\leq i \leq n$, and assume that for some $p \in \mathcal{P}(W^*(y))$ a non-zero spectral projection of $y$, we have $$(\partial_i y)^\sigma \# (py) = 0.$$
Finally, suppose that $y$ is an eigenvector of the number operator $\mathcal{N}_i$, i.e., $\mathcal{N}_i(y) = \partial_i y \# y_i = \lambda y$.
Then $\lambda = 0$ and $y \in \operatorname{Alg}(y_1, \ldots, \hat y_i, \ldots, y_n)$.

Moreover, if $(\partial_i y)^\sigma \# (py) = 0$ for each $1 \leq i \leq n$ and $y$ is an eigenvector of each $\mathcal{N}_1, \ldots, \mathcal{N}_n$, then $y \in \C$.
\end{lem}

\begin{proof}
As in the proof of Lemma \ref{lemma:zero-der}, we compute:
$$
0 = \tau\left(y_i(\partial_iy)^\sigma\#(py)\right)
= \tau\left( py(\partial_i y)\#y_i \right)
= \tau\left( py \mathcal{N}_i(y) \right)
= \lambda \tau\left( pyyp \right).
$$
If $\lambda \neq 0$, it follows that $pyyp = 0$, hence $py = 0$.
But then by Theorem \ref{thm:noZeroDivisors} we have that $y$ is constant.
Thus in either case, $y \in \alg(y_1, \ldots, \hat y_i, \ldots, y_n)$.
\end{proof}

\begin{thm}
Let $y=y^* \in \operatorname{Alg}(y_1,\dots,y_n) \subset W^*(y_1,\dots,y_m)=M$ be a non-constant polynomial in a $W^*$-probability space and assume that there exists a dual system to $y_1,\dots,y_m$, i.e., operators $R_j\in B(L^2(M))$, $1\leq j\leq m$ so that $[R_j,y_k]=\delta_{j=k} P_1$, where $P_1$ is the orthogonal projection onto $1\in L^2(M)$.
Suppose further that for each $j$, with $\mathcal{N}_j$ the number operator corresponding to $y_j$, we have $\mathcal{N}_j(y) = \lambda_jy$.
Then the spectral measure of $y$ is Lebesgue absolutely continuous.
\end{thm}

Note that the hypothesis that $y$ is an eigenvector of each $\mathcal{N}_j$ is satisfied when $y$ is homogeneous in each $y_j$, and in particular, whenever $y$ is a monomial.

\begin{proof}
We proceed much as in the proof of Proposition \ref{prop:nonsing}.
Assume to the contrary that the spectral measure of $y$ is not Lebesgue absolutely continuous.
As before, we apply Lemma \ref{lemma:sing-commute} to find $0\leq T_n \leq 1$ finite rank operators with $T_n\to p$ weakly for $p$ a spectral projection of $y$, and $\Vert [T_n,y]\Vert_1 \to 0$.


Fix $x \in M$.
We compute as follows:
\begin{align*}
0
&= \lim_{n\to\infty} \Vert Jx^*J R_i y \Vert_\infty \Vert [T_n, y] \Vert_1 \\
&\geq \lim_{n\to\infty} \Big| \operatorname{Tr}( Jx^*J R_i y [T_n, y]) \Big| \\
&= \lim_{n\to\infty} \Big| \operatorname{Tr}(Jx^*J [y, R_i] yT_n) \Big| \\
&= \lim_{n\to\infty} \Big| \operatorname{Tr}(Jx^*J ((\partial_i y)\# P_1) yT_n) \Big| \\
&= \Big| \operatorname{Tr}( Jx^*J P_1 ((\partial_i y)^\sigma \# (yp)) ) \Big| \\
&= \Big| \tau(((\partial_i y)^\sigma\#(yp))x) \Big|
\end{align*}
Here we are able to pass to the limit of $T_n$ as $P_1$ is finite rank.

We conclude that $(\partial_i y)^\sigma\#(yp) = 0$ as $\tau$ is faithful and $x$ was arbitrary.
Our stronger assumptions on $y$ now allow us to apply Lemma \ref{lemma:zero-der2}, and we see $y$ is constant.
\end{proof}

A related statement is the following.  Denote by $\vec{T}$ an $m$-tuple of operators $\vec{T}=(T_1,\dots,T_m)$.  Write $\vec{T}\geq \vec{T'}$ if $T_j \geq T_j'$ for all $j=1,\dots,m$.  Write $\mathscr{R}^m_{+}$ for the set of all $m$-tuples $\vec{T}$ satisfying $0\leq T_j \leq 1$.  

\begin{proposition}  Assume that $y_1,\dots,y_m \in M$ admits a dual system $R_1,\dots,R_m$.  Then $$\liminf_{\vec{T}\in \mathscr{R}^m_+} \left\Vert \sum [T_j,y_j]\right\Vert_1 > 0.$$
\end{proposition}
\begin{proof}
Assume for contradiction that we can find $\vec{T^n}\in \mathscr{R}^m_+$ satisfying $T_j\to 1$ and $\Vert \sum[T_j,y_j]\Vert_1 \to 0$.  Then
\begin{eqnarray*}
0&=&\lim_{n\to\infty} Tr\big(\sum_j R_i [T^n_j,y_j]\big) \\
   &=& \lim_{n\to\infty} Tr\big( [y_j,R_i] T^n_j \big) \\
   &=& \lim_{n\to\infty} Tr\big(P_1 T^n_i)\\ 
   & = & Tr(P_1) = 1
\end{eqnarray*}
which is a contradiction.
\end{proof}

\subsection{Decay of measure.}
\begin{lem}
\label{lem:coefficients}
Suppose that $y_1, \ldots, y_m \in W^*(y_1, \ldots, y_m) = M$ have a dual system of operators $R_1, \ldots, R_m \in B(L^2(M))$.
Assume that for some $P_{n}\in \alg(y_1, \ldots, y_m)$, there
exists a sequence of projections $f_{n}\in M$ so that $\Vert f_{n}\Vert_{\infty}=1$,
$\Vert f_{n}\Vert_{2}\to0$, $\Vert P_{n}f_{n}\Vert_{2}/\Vert f_{n}\Vert_{2}^{k}\to0$
for all $k\geq0$; moreover, assume that $p=\sup_{n}\deg P_{n}<+\infty$,
and for each monomial $y_{i_{1}}\cdots y_{i_{p}}$, its coefficient
$a_{i_{1}\cdots i_{p}}^{(n)}$ in $P_{n}$ is uniformly bounded in
$n$. Then all coefficients of monomials of degree $p$ in $P_{n}$
must converge to zero.\end{lem}
\begin{proof}
Let $P_{n}=u_{n}b_{n}$ be the polar decomposition of $P_{n}$, with
$u_{n}$ a unitary and $b_{n}=(P_{n}^{*}P_{n})^{1/2}$. Then, letting
$g_{n}=u_{n}f_{n}u_{n}^{*}$, we obtain another projection $g_{n}$ with
$\Vert g_{n}\Vert_{\infty}=1$, $\Vert g_{n}\Vert_{2}=\Vert f_{n}\Vert_{2}$
and 
\[
\Vert P_{n}f_{n}\Vert_{2}=\Vert u_{n}b_{n}f_{n}\Vert_{2}=\Vert f_{n}b_{n}u_{n}^{*}\Vert_{2}=\Vert f_{n}b_{n}\Vert_{2}=\Vert f_{n}u_{n}^{*}u_{n}b_{n}\Vert_{2}=\Vert u_{n}f_{n}u_{n}^{*}u_{n}b_{n}\Vert_{2}=\Vert g_{n}P_{n}\Vert_{2}.
\]
Thus, in particular, $\Vert g_{n}P_{n}\Vert_{2}/\Vert g_{n}\Vert_{2}^{k}\to0$
, for any $k\geq0$. 

Then we have:
\[
(g_{n}\otimes f_{n})\#(P_{n}\otimes1-1\otimes P_{n})=(g_{n}\otimes f_{n})\#(\sum_{i}\partial_{i}P_{n})\#(y_{i}\otimes1-1\otimes y_{i}),
\]
so that if we set $T_{i}^{n}=(g_{n}\otimes f_{n})\#\partial_{i}P_{n}$,
then
\[
\Vert\sum[T_{i}^{n},y_{i}]\Vert_{1}=\Vert g_{n}P_{n}\otimes f_{n}-g_{n}\otimes P_{n}f_{n}\Vert_{1}\leq\Vert P_{n}f_{n}\Vert_{2}\Vert g_{n}\Vert_{2}+\Vert g_{n}P_{n}\Vert_{2}\Vert f_{n}\Vert_{2}=2\Vert P_{n}f_{n}\Vert_{2}\Vert f_{n}\Vert_{2}.
\]
Thus for any $D\in B(H)$,
\[
|\sum Tr(T_{i}^{n}[y_{i},D])|=|Tr(\sum[T_{i}^{n},y_{i}]D)|\leq\Vert\sum[T_{i}^{n},y_{i}]\Vert_{1}\Vert D\Vert_{\infty}\leq2\Vert P_{n}f_{n}\Vert_{2}\Vert f_{n}\Vert_{2}\Vert D\Vert_{\infty}.
\]
Note that
\[
\Vert T_{i}^{n\dagger}\Vert_{\pi}\leq\Vert f_{n}\otimes g_{n}\Vert_{\pi}\Vert(\partial_{i}P_{n})^{\dagger}\Vert_{\pi}\leq C
\]
where $C=\sup_{n}\max_{i}\Vert(\partial_{i}P_{n})^{\dagger}\Vert_{\pi}<\infty$
by our assumptions of the uniform boundedness of both the degree of
$P_{n}$ and the coefficients of all monomials in $P_{n}$.
Let $D^{(n)}=T_{j}^{n\dagger\sigma}\#R_{j}$.
Then $\Vert D^{(n)}\Vert\leq\Vert T_{j}^{n\dagger}\Vert_{\pi}\Vert R_{j}\Vert_{\infty}=C\Vert R_{j}\Vert_{\infty}$.
Moreover, $[y_{i},D^{(n)}]=\delta_{i=j}T_{i}^{n\dagger}$. From this
one gets that
\[
\Vert T_{j}^{n}\Vert_{2}^{2}=|\langle T_{j}^{n},T_{j}^{n}\rangle|=|\sum_{i}Tr(T_{i}^{n}[y_{i},D])|\leq2C\Vert R_{j}\Vert_{\infty}\Vert P_{n}f_{n}\Vert_{2}\Vert f_{n}\Vert_{2}.
\]
Thus
\[
\Vert T_{j}^{n}\Vert_{2}\leq\sqrt{2\Vert R_{j}\Vert_{\infty}C}\Vert P_{n}f_{n}\Vert_{2}^{1/2}\Vert f_{n}\Vert_{2}^{1/2}.
\]
Let now 
\[
\Delta_{j}P_{n}=\tau(g_{n})^{-1}(\tau\otimes1)((g_{n}\otimes1)\#\partial_{j}P_{n})).
\]
 Then
\[
\deg\Delta_{j}P_{n}<\deg P_{n}
\]
and coefficients of all monomials in $\Delta_{j}P_{n}$ are uniformly bounded.
Moreover,
\begin{eqnarray*}
\Vert\Delta_{j}P_{n}f\Vert_{2} & = & \tau(g_{n})^{-1}\Vert\tau\otimes1((g_{n}\otimes f_{n})\#\partial_{j}P_{n}\Vert_{2}\\
 & \leq & \tau(g_{n})^{-1}\Vert(g_{n}\otimes f_{n})\#\partial_{j}P_{n}\Vert_{2}\\
 & \leq & \tau(g_{n})^{-1}\Vert T_{j}^{n}\Vert_{2}\\
 & \leq & \sqrt{2\Vert R_{j}\Vert_{\infty}C}\Vert P_{n}f_{n}\Vert_{2}^{1/2}\Vert f_{n}\Vert_{2}^{1/2}.
\end{eqnarray*}
It follows that
\[
\Vert\Delta_{j}P_{n}f_{n}\Vert_{2}/\Vert f_{n}\Vert_{2}^{k}\leq\sqrt{2\Vert R_{j}\Vert_{\infty}C}\cdot\Vert f_n\Vert_2^{1/2}\cdot\sqrt{\Vert P_{n}f_{n}\Vert_2/\Vert f_{n}\Vert_{2}^{2k}}\to0\qquad\textrm{as }n\to\infty.
\]

Thus $\Delta_{i}P_{n}$ once again satisfies the conditions of the
Lemma. Iterating, so does $\Delta_{i_{1}}\Delta_{i_{2}}\cdots\Delta_{i_{p}}P_{n}$,
for any sequence $i_{1},\dots,i_{p}$.

Let $p=\max_{n}\deg P_{n}$, and let $a_{i_{1}\cdots i_{p}}^{(n)}$
be the coefficient of the monomial $y_{i_{1}}\cdots y_{i_{p}}$ in
$P_{n}$. Then
\[
a_{i_{1}\cdots i_{p}}^{(n)}=\Delta_{i_{1}}\cdots\Delta_{i_{p}}P_{n},
\]
which means that
\[
|a_{i_{1}\dots i_{p}}^{(n)}|=\Vert a_{i_{1}\dots i_{p}}^{(n)}f_{n}\Vert_{2}/\Vert f_{n}\Vert_{2}\to0,
\]
as claimed.\end{proof}
\begin{lem}
\label{lem:polydecay}
Let $P\in\alg(y_1, \ldots, y_m)$ where again $y_1, \ldots, y_m$ admit a dual system.
Suppose that there exists a sequence of projections $f_{n}\in M$ and $t\in\mathbb{C}$ so that
$\Vert f_{n}\Vert_{\infty}=1$, $\Vert f_{n}\Vert_{2}\to0$, $\Vert Pf_{n}-tf_{n}\Vert_{2}/\Vert f_{n}\Vert_{2}^{k}\to0$
for all $k\geq0$. Then $P$ is constant.
\end{lem}
\begin{proof}
By replacing $P$ with $P-t1$ we may assume that $t=0$. Let $P_{n}=P$
be the constant sequence. Then the coefficient of any monomial of
maximal degree must be zero by Lemma \ref{lem:coefficients}. This means that $P=0$. \end{proof}

\begin{thm}
Let $y = y^* \in \alg(y_1, \ldots, y_n)$ be non-constant, where $y_1, \ldots, y_n$ admit a dual system, and let $\mu$ be the spectral measure of $y$.
Then for every $t \in \R$ there is some $\alpha > 0$ so that $\mu\paren{[t, t+\epsilon]} \leq \epsilon^\alpha$ for all $\epsilon > 0$ small enough.
\end{thm}

\begin{proof}
We will say that $\mu$ satisfies condition $P(t,\alpha)$, $\alpha>0$ if the inequality
\[
\mu[t,t+\epsilon]\leq\epsilon^{\alpha}
\]
holds for sufficiently small $\epsilon$. Equivalently, we require
that
\[
\log\mu[t,t+\epsilon]\leq\alpha\log\epsilon
\]
for sufficiently small $\epsilon$. 

Fix $t\in\R$, and assume to the contrary $P(t,\alpha)$ does not hold for any $\alpha>0$.
Then there exists a sequence $\epsilon_{n}\to0$ with the property
that
\[
\frac{\log\mu[t,t+\epsilon_{n}]}{\log\epsilon_{n}}\to0.
\]

Put $\lambda_{n}=\mu[t,t+\epsilon_{n}]^{1/2}$. Then
$
\frac{\log\lambda_{n}}{\log\epsilon_{n}}\to0
$. 
Thus
\[
\frac{(1-k)\log\lambda_{n}+\log\epsilon_{n}}{\log\epsilon_{n}}\to1.
\]
Thus
\[
(1-k)\log\lambda_{n}+\log\epsilon_{n}\to-\infty.
\]
Exponentiating, we get
\[
\epsilon_{n}\lambda^{1-k}\to0
\]

Let $f_{n}=\chi_{[t,t+\epsilon]}(y)$. Then $\Vert f_{n}\Vert_{\infty}=1$
and $\Vert f_{n}\Vert_{2}=\lambda_{n}$.
It follows that
\[
\Vert yf_{n}-tf_{n}\Vert_{2}\lesssim\epsilon_{n}\lambda_{n}
\]
and so
\[
\Vert yf_{n}-tf_{n}\Vert_{2}/\Vert f_{n}\Vert_{2}^k\lesssim\epsilon_{n}\lambda_{n}^{1-k}\to0.
\]
Lemma \ref{lem:polydecay} now implies that $y$ is constant, a contradiction.
\end{proof}


\begin{thebibliography}{BCG03}

\bibitem[A14]{anderson:14} G.~W.~Anderson, {\em Preservation of algebraicity in free probability}, Preprint, 2014, ArXiV:1406.6664.

\bibitem[AZ06]{anderson-zeitouni:06} G.~W.~Anderson and O.~Zeitouni, {\em A law of large numbers for finite-range dependent random matrices}, Comm.  Pure  App. Math, Vol. \textbf{61} (2008), 1118Ð1154.

\bibitem[BCG03]{bcg} P. Biane, M. Capitaine and A. Guionnet, {\em
Large deviation bounds for matrix {B}rownian motion}, {Invent.
Math.}, \textbf{152} (2003), 433--459.

\bibitem[CS05]{cs} { A. Connes and D. Shlyakhtenko}, {\em $L^{2}$-homology
for {von Neumann} algebras}, {J. Reine Angew. Math.}, \textbf{586}
(2005) 125--168.

\bibitem[Dab10]{yoann:SPDE}Y. Dabrowski, A free stochastic partial
differential equation, Annal. IHP \textbf{50}, 2014, 1404-1455.

\bibitem[GF14]{guionnet-figalli:transport}A. Figalli and A. Guionnet,
\emph{Universality in several-matrix models via approximate transport
maps}, Preprint, 2014, ArXiV:1407.2759

\bibitem[GS12]{guionnet-shlyakhtenko:transport}A. Guionnet and D.
Shlyakhtenko, \emph{Free monotone transport}, Invent. Math. \textbf{197} (2014), 613-661.

\bibitem[Kat66]{kato:perturb} T. Kato, Perturbation theory of linear operators, Springer, 1966.

\bibitem[MSW15]{msw-noatoms}, {T. Mai, R. Speicher and M. Weber},
\textit{Absence of algebraic relations and of zero divisors under the assumption of full non-microstates free entropy dimension},
Preprint, 2015.

\bibitem[SS13]{shlyakhtenko-skoufranis:atoms}D. Shlyakhtenko and
P. Skoufranis, \emph{Freely independent random variables with non-atomic
distributions}, Trans. AMS \textbf{367} (2015), 6267-6291.

\bibitem[Voi79]{dvv:norm-ideal-perturb} D. Voiculescu, {\em Some results on norm-ideal perturbations of Hilbert space operators}, J. Operator Theory {\bf 2} (1979), 3--37.

\bibitem[Voi93]{dvv:entropy1}D. Voiculescu, \emph{The analogues of entropy and of Fisher's information measure in free probability theory, I}, Comm. Math. Phys. \textbf{155} (1993) 71\textendash{}92.

\bibitem[Voi94]{dvv:entropy2}D. Voiculescu, \emph{The analogues of
entropy and of Fisher's information measure in free probability theory,
II}, Invent. Math. \textbf{118} (1994) 411\textendash{}440.

\bibitem[Voi98]{dvv:entropy5}D. Voiculescu, \emph{The analogues of
entropy and of Fisher's information measure in free probability, V},
Invent. Math., \textbf{132} (1998) 189-227.\end{thebibliography}
\end{document}